\documentclass[11pt]{amsart}
\usepackage{amssymb, amsmath, amsthm}
\usepackage[colorlinks=true,linkcolor=blue,citecolor=red]{hyperref}
\usepackage{color}
\usepackage{epsfig}
\usepackage{amsmath}
\usepackage{amssymb}
\usepackage{amscd}
\usepackage{graphicx}
\usepackage{mathrsfs}
\usepackage{subfigure}
\usepackage{tikz}
\usepackage{epstopdf}
\usepackage{bm}
\usepackage[T1]{fontenc}
\numberwithin{equation}{section}
\newtheorem{lemma}{Lemma}[section]
\newtheorem{thm}{Theorem}

\newtheorem{prop}{Proposition}[section]
\newtheorem{RHP}{Riemann-Hilbert Problem}

\usepackage{epsfig}

\numberwithin{equation}{section}
\numberwithin{thm}{section}
\numberwithin{prop}{section}
\numberwithin{lemma}{section}
\numberwithin{re}{section}
\numberwithin{coro}{section}



\thanks{ Email: sftian@cumt.edu.cn (S.F. Tian). }

\begin{document}

\title[The long-time asymptotic behaviors of the solutions for the AB system]{The long-time asymptotic behaviors of the solutions for the coupled dispersive AB system with weighted Sobolev initial data}


\author[Wang]{Zi-Yi Wang}

\author[Tian]{Shou-Fu Tian$^{*}$}

\address{Shou-Fu Tian (Corresponding author$^{*}$)\newline
School of Mathematics and Institute of Mathematical Physics, China University of Mining and Technology, Xuzhou 221116, P. R. China}
\email{sftian@cumt.edu.cn}

\author[Li]{Zhi-Qiang Li}

\begin{abstract}
{In this work, we employ the $\bar{\partial}$-steepset descent method to study the Cauchy problem of the coupled dispersive AB system with initial conditions in weighted Sobolev space $H^{1,1}(\mathbb{R})$,
 \begin{align*}
\left\{\begin{aligned}
&A_{xt}-\alpha A-\beta AB=0,\\
&B_{x}+\frac{\gamma}{2}(|A|^2)_t=0,\\
&A(x,0)=A_0(x),~~~~B(x,0)=B_0(x)\in H^{1,1}(\mathbb{R}).
\end{aligned}\right.
\end{align*}
 Begin with the Lax pair of the coupled dispersive AB system, we successfully derive the solutions of the coupled dispersive AB system by constructing the basic Riemann-Hilbert problem. By using the $\bar{\partial}$-steepset descent method, the long-time asymptotic behaviors of the solutions for the coupled dispersive AB system are characterized without discrete spectrum. Our results demonstrate that compared with the previous results, we increase the accuracy of the long-time asymptotic solution from $O(t^{-1}\log t)$ to $O(t^{-1})$. }
\end{abstract}

\maketitle
\tableofcontents

\section{Introduction}

It is understood that in the geophysical fluid dynamics, the process of converting the available potential energy of the rotating layered fluid into a growing disturbance is referred to as the  Baroclinic instability\cite{boo-1,boo-2,boo-3}. The AB system is one of the most universal baroclinic wave packets equations, which was firstly proposed by Pedlosky through employing the singular perturbation theory\cite{boo-1}. In particular, the AB system is an important model in  geophysical fluid due to its surprisingly rich physical significance and broad applicability to a number of different areas. The AB system has also been suggested as a model for describing the unstable process of baroclinic waves in geophysical flow and describing ultra-short optical pulse propagation in nonlinear optics\cite{boo-2,boo-3}. The coupled dispersive AB system reads
\begin{align}\label{AB}
\left\{\begin{aligned}
&A_{xt}-\alpha A-\beta AB=0,\\
&B_{x}+\frac{\gamma}{2}(|A|^2)_t=0,
\end{aligned}\right.
\end{align}
where $A=A(x,t)$ is complex function, while $B=B(x,t)$ is real function, $\alpha>0$ and $\alpha<0$ represent the supercritical and subcritical case of the shear, respectively, and the real parameters $\beta$, $\gamma$ denote the nonlinearities. In this work, we just consider the case that $\beta\gamma=-1$. The research for AB system has been studied extensively during the past years and a plethora of results are available. The inverse scattering method was developed rigorously by Gibbion et al. to present the $N$-soliton solutions of the AB system\cite{boo-4}. Further the periodic solutions and Whitham equations for the AB system were studied by Kamchatnov et al.\cite{boo-5}, while the envelope solitary waves and periodic waves were further detailed analyzed by Tan et al.\cite{boo-6}. Moreover, the classical Darboux transformation method was developed to present the soliton and breather solutions of the AB system\cite{boo-7}. Those results were generalized by introducing the generalized Darboux transformation to seek the high-order rogue wave solutions\cite{boo-8,boo-9}. In addition, the modulation periodic wave solutions along with the corresponding Whitham equations were also derived via Whitham modulation theory\cite{boo-10}. Following a long development involving many researchers, for the problem where the long-time asymptotic of the solutions of the AB system as $t\rightarrow\infty$, the nonlinear steepset descent method has been introduced in\cite{boo-11}. However, all the researches did not offer any insight about the Cauchy problem of the coupled AB system with initial condition in weighted Sobolev space $H^{1,1}(\mathbb{R})$.

The aim of the present work is to address this challenge and characterize the long-time asymptotic behaviors for the coupled dispersive AB system with the initial value condition
\begin{align}\label{IVC}
A(x,0)=A_0(x),~~~~B(x,0)=B_0(x)\in H^{1,1}(\mathbb{R}),
\end{align}
where the weighted Sobolev space is defined as
\begin{align}
\begin{split}
&H^{1,1}(\mathbb{R})=L^{2,1}(\mathbb{R})\cap H^1(\mathbb{R}),\\
&L^{2,1}(\mathbb{R})=\{(1+|\cdot|^2)^{\frac{1}{2}}f\in L^{2}(\mathbb{R})\},\\
&H^1(\mathbb{R})=\{f\in L^{2}(\mathbb{R})|f'\in L^{2}(\mathbb{R})\}.
\end{split}
\end{align}
We do so by using $\bar{\partial}$ steepest descent method. As far as we know,  in 1974 the long-time asymptotic behavior of equations was first noticed by Manakov\cite{boo-12}. Latter, in 1979 Zakharov and Manakov used the inverse scattering method to give the first long-time asymptotic solution of the NLS equation with initial value decay \cite{boo-13}. The nonlinear steepest descent method was first used to study the long-time asymptotic behavior of the solution for MKdV equation by Defit and Zhou\cite{boo-14}, and was subsequently extended and employed in numerous works, including the study of long-time asymptotic behavior of the NLS equation\cite{boo-15}, KdV equation\cite{boo-16}, Fokas-Lenells equation\cite{boo-17}, coupled NLS equation\cite{boo-18}, the Sasa-Satsuma equation\cite{boo-19}, and other works\cite{boo-20,boo-21,boo-22,boo-23}.
However, the nonlinear steepest descent method need to take the delicate estimates which involving $L_p$ estimates of Cauchy projection operators into consider, the $\bar{\partial}$ steepest descent method was introduced to overcome this challenge, which was first proposed by McLaughlin and Miller\cite{boo-24,boo-25}. Hereafter, the method was also developed rigorously  to study the defocusing NLS under essentially minimal regularity assumptions on finite mass initial data\cite{boo-26} and the defocusing NLS with finite density initial data\cite{boo-27}. With the continuous development of the $\bar{\partial}$ steepest descent method, more and more work has been studied, including KdV equation\cite{boo-28}, focusing NLS equation\cite{boo-29}, short-pluse equation\cite{boo-30}, focusing Fokas-Lenells equation\cite{boo-31}, modified Camassa-Holm equation\cite{boo-32}, Wadati-Konno-Ichikawa\cite{boo-33} and other works\cite{boo-34,boo-35,boo-36,boo-37}.

In this work, we characterize the long time asymptotic behaviors of the solution for the coupled dispersive AB system. As far as we know, Reference\cite{boo-11} has employed the Defit-Zhou steepest descent method to present the long time asymptotics of the coupled dispersive AB system with the initial value $A_0(x),B_{0}(x)\in\mathcal{S}(\mathbb{R})$, which defined as $\mathcal{S}(\mathbb{R})=\{f(x)\in\mathbb{C}^{\infty}(\mathbb{R})|x^\alpha f^{(\beta)}(x)\in L^{\infty}(\mathbb{R}),\alpha,\beta\in\mathbb{Z}^+\}$. The results demonstrate that
\begin{align}
&A=-8Re\bigg(\sqrt{-\frac{\lambda_0}{\alpha t}}(\delta_{\lambda_0}^{0})^2i\beta^X(\lambda_0)\bigg)+O\bigg(\frac{\log t}{t}\bigg),\notag\\
&B=O\bigg(\frac{\log t}{t}\bigg),
\end{align}
where $\lambda_0=\sqrt{\frac{-\alpha t}{4x}}$ is a fixed constant. In our result, we provide the first characterization of the long-time asymptotic behaviors of the solutions for the coupled dispersive AB system with the initial value $A_0(x),B_{0}(x)\in H^{1,1}(\mathbb{R})$.

The rest of this work is organized as follows. Begin with the Lax pair of the AB system, we present the asymptotic property, analyticity and the symmetry property of the eigenfunction in section~\ref{s:2}. In section~\ref{s:3}, we construct the Riemann-Hilbert(RH) problem of the coupled dispersive AB system. The detailed construction process can be found in reference\cite{boo-11}. We introduce the function $\delta(z)$ to construct a new RH problem for $M^{(1)}(z)$ whose jump matrix can be triangulated at phase points $\pm z_0$ in section~\ref{s:4}. Additionally, in order to make continuous extension to the jump matrix of the RH problem for $M^{(1)}(z)$, we introduce the matrix value function $R^{(2)}(z)$ and define a new mixed $\bar{\partial}$-RH problem for $M^{(2)}(z)$.  In section~\ref{s:5}, we divide the $M^{(2)}(z)$ into two parts, including the model RH problem for $M^{(2)}_{RHP}(z)$ with $\bar{\partial}R^{(2)}=0$ and the pure $\bar{\partial}$-RH problem for $M^{(3)}(z)$ with $\bar{\partial}R^{(2)}\neq0$. Further, we consider two scaling transformations to solve $M^{(2)}_{RHP}(z)$. In addition, we show the long-time asymptotic behavior of the $M^{(3)}(z)$. In the end, we characterize the long-time asymptotic behaviors of the coupled dispersive AB system with the aid of the above results in section~\ref{s:6}.
\section{Spectral Analysis}\label{s:2}

In this section, we will focus on constructing the basic Riemann-Hilbert of the system\eqref{AB}. Firstly, we present the spectral analysis of the system\eqref{AB}.
\subsection{Lax pair}
The Lax pair of system\eqref{AB} reads
\begin{align}
\begin{split}
\varphi_x=\tilde{X}\varphi,\\\label{Lax1}
\varphi_t=\tilde{T}\varphi,
\end{split}
\end{align}
where
\begin{align}
\tilde{X}=\left(
    \begin{array}{cc}
      -iz & \frac{\sqrt{\beta\gamma}}{2}A \\
      -\frac{\sqrt{\beta\gamma}}{2}\bar{A} & iz \\
    \end{array}
  \right),~~
\tilde{T}=\left(
    \begin{array}{cc}
      \frac{i(\alpha+\beta B)}{4z} & -\frac{i\sqrt{\beta\gamma}}{4z}A_t \\
      -\frac{i\sqrt{\beta\gamma}}{4z}\bar{A}_t & \frac{-i(\alpha+\beta B)}{4z} \\
    \end{array}
  \right).
\end{align}

In the case of $\beta\gamma=-1$, the Lax pair~\eqref{Lax1} turn into
\begin{align}
\begin{split}
\varphi_x+iz\sigma_3\varphi=X\varphi,\\\label{Lax2}
\varphi_t-\frac{i\alpha}{4z}\sigma_3\varphi=T\varphi,
\end{split}
\end{align}
where
\begin{align}
X=\frac{i}{2}\left(
    \begin{array}{cc}
      0 & A \\
      -\bar{A} & 0 \\
    \end{array}
  \right),~~
T=\frac{1}{4z}\left(
    \begin{array}{cc}
      i\beta B & A_t \\
      \bar{A}_t & -i\beta B \\
    \end{array}
  \right).
\end{align}

\textbf{Case I:} $z\rightarrow\infty$.

In the current case, we choose $\psi(x,t,z)=\varphi(x,t,z)e^{(izx-\frac{i\alpha}{4z}t)\sigma_3}$, then the Lax pair\eqref{Lax2} change into
\begin{align}
\begin{split}
&\psi_x+iz[\sigma_3,\psi]=X\psi,\\\label{Lax3}
&\psi_t-\frac{i\alpha}{4z}[\sigma_3,\psi]=T\psi,
\end{split}
\end{align}
which can be written in full derivative form
\begin{align}\label{4}
d(e^{i(zx-\frac{\alpha}{4z}t)\hat{\sigma}_3}\psi)=e^{i(zx-\frac{\alpha}{4z}t)\hat{\sigma}_3}U(x,t,z)\psi(x,t,z),
\end{align}
where $[\sigma_3,\psi]=\sigma_3\psi-\psi\sigma_3$, $U(x,t,z)=Xdx+Tdt$, $e^{\hat{\sigma}_{3}}A=e^{\sigma_{3}}Ae^{-\sigma_{3}}$ with a $2\times2$ matrix $A$.

\subsection{Asymptotic Property and Analyticity}
In order to construct the RH problem of the solutions for Lax pair\eqref{Lax3}, we need to obtain the solutions of the Lax pair\eqref{Lax3} which approach the $2\times2$ identity matrix as $z\rightarrow\infty$. Thus, we consider the asymptotic expansion of a solution of Lax pair as $z\rightarrow\infty$.

\begin{align}\label{zhankai}
\psi(x,t,z)=\psi^{(0)}(x,t)+\frac{\psi^{(1)}(x,t)}{z}+\frac{\psi^{(2)}(x,t)}{z^2}
+\frac{\psi^{(3)}(x,t)}{z^3}+O(\frac{1}{z^4}),
\end{align}
where $\psi^{(j)}(x,t)$ are independent of $z$.

Substituting \eqref{zhankai} into the Lax pair\eqref{Lax3} and comparing the same power of $z$ yields that $\psi^{(0)}(x,t)$ is the diagonal matrix and $\psi^{(0)}_x(x,t)=\psi^{(0)}_t(x,t)=0$. Further we obtain that $\psi^{(0)}(x,t)$ is a constant matrix and satisfy
\begin{align}
\mathbb{I}=\lim_{z\rightarrow\infty}\lim_{x\rightarrow\infty}\psi(x,t,z)
=\lim_{x\rightarrow\infty}\lim_{z\rightarrow\infty}\psi(x,t,z)
=\lim_{x\rightarrow\infty}\psi^{(0)}(x,t)=\psi^{(0)}(x,t).
\end{align}

The solutions of \eqref{Lax3} can be derived as Volterra integrals
\begin{align}
\begin{split}
\psi_-(x,t,z)=\mathbb{I}+\int_{-\infty}^{x}e^{iz(y-x)\hat{\sigma}_3}X(y,t,z)\psi_-(x,t)dy,\\
\psi_+(x,t,z)=\mathbb{I}-\int_{x}^{+\infty}e^{iz(y-x)\hat{\sigma}_3}X(y,t,z)\psi_+(x,t)dy,
\end{split}
\end{align}
from which we can derive the analytical of $\psi_{\pm}(x,t,z)$.

\begin{prop}\label{5}
$\psi_{\pm}(x,t,z)$ satisfy the following properties:
\begin{itemize}
  \item $\psi^1_-(x,t,z)$ is analytic in $Im~z>0$ and $\psi^1_-(x,t,z)=\left(
                                                                               \begin{array}{c}
                                                                                 1 \\
                                                                                 0 \\
                                                                               \end{array}
                                                                             \right)+O(\frac{1}{z})$ as
  $z\rightarrow\infty$ and $Im~z\geq0$.
  \item $\psi^2_-(x,t,z)$ is analytic in $Im~z<0$ and $\psi^2_-(x,t,z)=\left(
                                                                               \begin{array}{c}
                                                                                 0 \\
                                                                                 1 \\
                                                                               \end{array}
                                                                             \right)+O(\frac{1}{z})$ as
  $z\rightarrow\infty$ and $Im~z\leq0$.
  \item $\psi^1_+(x,t,z)$ is analytic in $Im~z<0$ and $\psi^1_+(x,t,z)=\left(
                                                                               \begin{array}{c}
                                                                                 1 \\
                                                                                 0 \\
                                                                               \end{array}
                                                                             \right)+O(\frac{1}{z})$ as
  $z\rightarrow\infty$ and $Im~z\leq0$.
  \item $\psi^2_+(x,t,z)$ is analytic in $Im~z>0$ and $\psi^2_+(x,t,z)=\left(
                                                                               \begin{array}{c}
                                                                                 0 \\
                                                                                 1 \\
                                                                               \end{array}
                                                                             \right)+O(\frac{1}{z})$ as
  $z\rightarrow\infty$ and $Im~z\geq0$,
\end{itemize}
where $\psi^1_{\pm}(x,t,z)$ and $\psi^2_{\pm}(x,t,z)$ represent the first and second columns of $\psi_{\pm}(x,t,z)$, respectively.
\end{prop}

\subsection{Scattering Matrix and Reflection Coefficient}

Since the solution $\psi_{\pm}(x,t,z)$ solve \eqref{Lax3}, there is a $2\times2$ matrix $S(z)=(s_{ij})_{2\times2}$ which is independent on $x$ and $t$ satisfies the following proposition.

\begin{prop}
There is a $2\times2$ matrix $S(z)=(s_{ij})_{2\times2}$ which is independent on $x$ and $t$ such that $\psi_{\pm}(x,t,z)$ satisfy
\begin{align}\label{8}
\psi_-(x,t,z)=\psi_+(x,t,z)e^{-it\theta(z)\hat{\sigma}_3}S(z),~~\theta(z)=z\frac{x}{t}-\frac{\alpha}{4z},
~~z\in\mathbb{R},
\end{align}
where $s_{ij}(z)$ are called scattering coefficients and can be expressed as
\begin{align*}
&s_{11}(z)=\det(\psi_-^1(x,t,z),\psi_+^2(x,t,z)),&s_{22}&=\det(\psi_+^1(x,t,z),\psi_-^2(x,t,z)),\\
&s_{12}(z)=\det(\psi_-^1(x,t,z),\psi_+^1(x,t,z)),&s_{21}&=-\det(\psi_+^2(x,t,z),\psi_-^2(x,t,z)).
\end{align*}
The analyticity of scattering coefficients are given with the aid of Proposition~\ref{5}, namely, $s_{11}(z)$ is analytic in $Im~z>0$ and $s_{22}(z)$ is analytic in $Im~z<0$, whereas $s_{12}(z)$ and $s_{21}(z)$ are not analytic.
\end{prop}

In addition, we define the reflection coefficient
\begin{align}
r(z)=\frac{s_{21}(z)}{s_{11}(z)}.
\end{align}

\subsection{Symmetry Property}
\begin{prop}\label{prop2.3}
Solutions $\psi_{\pm}(x,t,z)$, scattering matrix $S(z)$ and reflection coefficient $r(z)$ satisfy the following symmetry
\begin{align}
\psi_{\pm}(x,t,k)=\sigma_1\overline{\psi_{\pm}(x,t,\bar{z})}\sigma_1,~~S(z)=\sigma_1\overline{S(\bar{z})}\sigma_1,
~~r(z)=\sigma_1\overline{r(\bar{z})}\sigma_1,
\end{align}
where $\sigma_1=\left(
                  \begin{array}{cc}
                    0 & 1 \\
                    1 & 0 \\
                  \end{array}
                \right)
$.
\end{prop}

\textbf{Case II:} $z\rightarrow0$.

In the case of $z\rightarrow0$, we choose $\psi^0(x,t,z)=\varphi(x,t,z)e^{(izx-\frac{i\alpha}{4z}t)\sigma_3}$, then the Lax pair~\eqref{Lax2} change into
\begin{align}
\begin{split}
&\psi^0_x+iz[\sigma_3,\psi^0]=X\psi^0,\\\label{Lax3}
&\psi^0_t-\frac{i\alpha}{4z}[\sigma_3,\psi^0]=T\psi^0,
\end{split}
\end{align}
which can be written in full derivative form
\begin{align}\label{4}
d(e^{i(zx-\frac{\alpha}{4z}t)\hat{\sigma}_3}\psi^0(x,t,z))
=e^{i(zx-\frac{\alpha}{4z}t)\hat{\sigma}_3}U(x,t,z)\psi^0(x,t,z),
\end{align}
where $[\sigma_3,\psi^0]=\sigma_3\psi^0-\psi^0\sigma_3$, $U(x,t,z)=Xdx+Tdt$, $e^{\hat{\sigma}_{3}}A=e^{\sigma_{3}}Ae^{-\sigma_{3}}$ with a $2\times2$ matrix $A$.

In the current case, we consider the following asymptotic expansion
\begin{align}\label{expansion1}
\psi^0(x,t,z)=\psi_0^{(0)}(x,t)+\psi_0^{(1)}(x,t)z+\psi_0^{(2)}(x,t)z^2
+O(z^3).
\end{align}

We substitute \eqref{expansion1} into Lax pair\eqref{Lax3} and compare the same power of $z$ derive that
\begin{align*}
&\psi_{0,x}^{(0)}=X\psi_0^{(0)},~~\psi_{0,t}^{(0)}-\frac{i\alpha}{4}[\sigma_3,\psi_0^{(1)}]=zT\psi_0^{(1)},
-\frac{i\alpha}{4}[\sigma_3,\psi_0^{(0)}]=zT\psi_0^{(0)},\\
&\psi_{1,x}^{(0)}+i[\sigma_3,\psi_0^{(0)}]=X\psi_0^{(1)},~~
\psi_{1,t}^{(0)}-\frac{i\alpha}{4}[\sigma_3,\psi_0^{(2)}]=zT\psi_0^{(2)}.
\end{align*}

\textbf{Remark 1:} We do the same transformation as $z\rightarrow0$ and $z\rightarrow\infty$, so the analytically and the symmetry property of $\psi^0(x,t,z)$ as $z\rightarrow0$ are same with the case of $z\rightarrow\infty$, for detail see the case $z\rightarrow\infty$.

\textbf{Remark 2:} We ignore the analysis near the neighborhood of $z=0$ and refer to Proposition~\ref{prop4.1} for further explanation.

In addition, we give the following result.
\begin{lemma}
The reflection coefficient $r(z)\in H^{1,1}(\mathbb{R})$ provided that the initial datum $A_0(x), B_0(x)\in H^{1,1}(\mathbb{R})$.
\end{lemma}

\section{Construction of the basic RH Problem}\label{s:3}
In combination with the above properties, we propose the basic RH problem.

Define a sectionally meromorphic matrix

\begin{align}\label{6}
M(x,t;z)=\left\{\begin{aligned}
&M_{+}(x,t;z)=\left(\frac{\psi_-^1}{s_{11}},\psi_+^2\right), \quad z\in\{z\in\mathbb{C}|Im~z>0\},\\
&M_{-}(x,t;z)=\left(\psi_+^1,\frac{\psi_-^2}{s_{22}}\right), \quad z\in\{z\in\mathbb{C}|Im~z<0\},
\end{aligned}\right.
\end{align}
where
\begin{align}
M_{\pm}(x,t,z)=\lim_{\varepsilon\rightarrow0^{+}}M(x,t,z\pm i\varepsilon),~~\varepsilon\in\mathbb{R}.
\end{align}
Based on Proposition~\ref{prop2.3}, we derive that $s_{11}(z)\neq0$ for all $z\in\mathbb{R}\cup\{z\in\mathbb{C}|Im~z>0\}$. Further the RH problem is presented as following:
\begin{RHP}\label{RHP1}
Find a matrix value function $M(z)$ admits:
\begin{itemize}
  \item Analyticity:~$M(x,t,z)$ is analytic in $\mathbb{C}\setminus\mathbb{R}$;
  \item Jump condition:
  \begin{align}
  M_{+}(x,t,z)=M_{-}(x,t,z)J(x,t,z),~~~z\in\mathbb{R},
  \end{align}
  where
  \begin{align}\label{7}
  J(x,t,z)=e^{-it\theta\hat{\sigma}_3}\left(\begin{array}{cc}
                   1-|r(z)|^2 & -\overline{r(\bar{z})} \\
                   r(z) & 1
                 \end{array}\right);
  \end{align}
  \item Asymptotic behavior:
  \begin{align}
  M(x,t,z)\rightarrow\mathbb{I}~~as~~z\rightarrow\infty.
  \end{align}
\end{itemize}
\end{RHP}

Furthermore, in terms of the expansion~\eqref{zhankai} of $\psi(x,t,z)$, we obtain
\begin{gather*}
\psi^{(0)}_x+i[\sigma_3,\psi^{(1)}]=X\psi^{(0)},\\
\psi^{(1)}_t-i\frac{\alpha}{4}[\sigma_3,\psi^{(1)}]=zT\psi^{(0)}.
\end{gather*}
Thus, we conclude
\begin{align}
&A=4\lim_{z\rightarrow\infty}\big(zM\big)_{12},\\
&B=-\frac{4i}{\beta}\lim_{z\rightarrow\infty}\frac{d}{dt}\big(zM\big)_{11}.
\end{align}

\bigskip
\
\section{Conjugation}\label{s:4}

In the scheme of $\alpha x>0$, we choose $\alpha<0$ and $x<0$. It can be found that the solutions of the initial value problem is rapidly decayed\cite{boo-11}. Now, we consider that $\alpha x<0$. In the current case, we set $\alpha<0$ and $x>0$. (Taking the similar method can obtain the result under the case of $\alpha>0$ and $x<0$.)  We rewrite the oscillation term $e^{2it\theta}$ as $e^{t\vartheta}$ where $\vartheta(z)=2i(\frac{x}{t}z-\frac{\alpha}{4z})$. And $\theta(z)$ can be rewritten as $\theta(z)=-\frac{\alpha z}{4}(\frac{1}{z_0^2}+\frac{1}{z^2})$. Further we obtain the two critical points $\pm z_0=\pm\sqrt{-\frac{\alpha t}{4x}}$. Thus we can derive the decaying domains of the oscillation term, which are shown in Fig~\ref{fig1}. In this work, we only consider the case of $t\rightarrow+\infty$.

\begin{center}
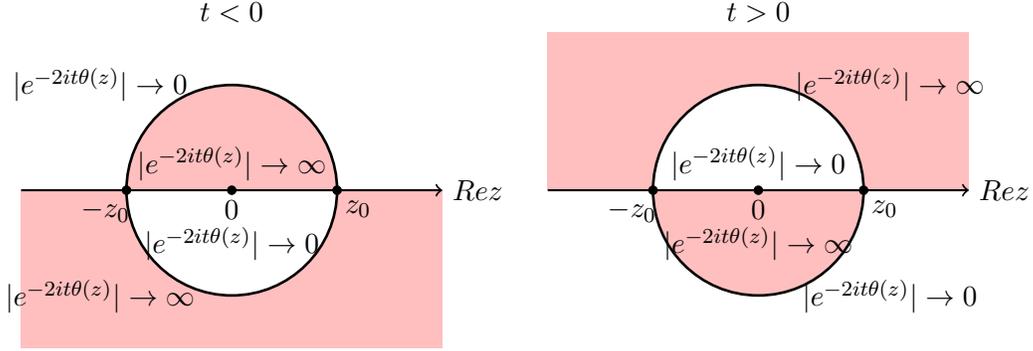
\begin{figure}
\begin{tikzpicture}[scale=0.7]
\path [fill=pink] (-1,0) -- (-9,0) to (-9,-3) -- (-1,-3);
\filldraw[pink, line width=0.5](-3,0) arc (0:180:2);
\filldraw[white, line width=0.5](-7,0) arc (-180:0:2);
\draw[->][thick](-9,0)--(-1,0);
\draw[fill] (-1,0)node[right]{$Rez$};
\draw[fill] (-5,0)node[below]{$0$} circle [radius=0.08];
\draw[fill] (-7.4,0)node[below]{$-z_{0}$};
\draw[fill] (-2.6,0)node[below]{$z_{0}$};
\draw[fill] (-7,0)circle [radius=0.08];
\draw[fill] (-3,0)circle [radius=0.08];
\draw[fill] (-7.5,1.5)node[above]{$|e^{-2it\theta(z)}|\rightarrow0$};
\draw[fill] (-5,0)node[above]{$|e^{-2it\theta(z)}|\rightarrow\infty$};
\draw[fill] (-5,-0.5)node[below]{$|e^{-2it\theta(z)}|\rightarrow0$};
\draw[fill] (-7.5,-1.5)node[below]{$|e^{-2it\theta(z)}|\rightarrow\infty$};
\draw[fill] (-5,3)node[above]{$t<0$};
\path [fill=pink] (1,0) -- (9,0) to (9,3) -- (1,3);
\filldraw[white, line width=0.5](7,0) arc (0:180:2);
\filldraw[pink, line width=0.5](3,0) arc (-180:0:2);
\draw[->][thick](1,0)--(9,0);
\draw[fill] (9,0)node[right]{$Rez$};
\draw[fill] (5,0)node[below]{0} circle [radius=0.08];
\draw[fill] (7.4,0)node[below]{$z_{0}$} ;
\draw[fill] (2.6,0)node[below]{$-z_{0}$} ;
\draw[fill] (7,0)circle [radius=0.08];
\draw[fill] (3,0)circle [radius=0.08];
\draw[fill] (7.5,1.5)node[above]{$|e^{-2it\theta(z)}|\rightarrow\infty$};
\draw[fill] (5,0)node[above]{$|e^{-2it\theta(z)}|\rightarrow0$};
\draw[fill] (5,-0.5)node[below]{$|e^{-2it\theta(z)}|\rightarrow\infty$};
\draw[fill] (7.5,-1.5)node[below]{$|e^{-2it\theta(z)}|\rightarrow0$};
\draw[fill] (5,3)node[above]{$t>0$};
\draw(-5,0) [black, line width=1] circle(2);
\draw(5,0) [black, line width=1] circle(2);
\end{tikzpicture}
\caption{Exponential decaying domains.}\label{fig1}
\end{figure}
\end{center}

We introduce a notation
\begin{align}
\upsilon(s)=-\frac{1}{2\pi}\log(1-|r(s)|^2),
\end{align}
and a function
\begin{align}
\delta(z)=\exp\bigg(i\int_{-z_0}^{z_0}\frac{\upsilon(s)}{s-z}ds\bigg),
\end{align}
which satisfies the following properties:
\begin{itemize}
  \item $\delta(z)$ is analytic in $\mathbb{C}\backslash(-z_0,z_0)$;
  \item The boundary values $\delta_{\pm}(z)$ satisfy that
  \begin{align}
  \delta_+(z)=(1-|r|^2)\delta_-(z),~~as~z\in(-z_0,z_0);
  \end{align}
  \item
\begin{align}
\delta(z)=1-\frac{i}{z}\int_{-z_0}^{z_{0}}\upsilon(s)ds+O(z^{-2}),
~~as~~|z|\rightarrow\infty,~~|\arg(z)|\leq c<\pi.
\end{align}
  \item  As $z\rightarrow \pm z_{0}$ along ray $z=\pm z_{0}+e^{i\psi}l$ with $|\psi|\leq c<\pi$, $l>0$, we have
\begin{align}
|\delta(z)-\delta_0(\pm z_0)(z\mp z_0)^{\pm i\upsilon(\pm z_0)}|\leq C\parallel r\parallel_{H^{1}(\mathbb{R})}|z\mp z_{0}|^{\frac{1}{2}},
\end{align}
where
\begin{align}
&\delta_0(\pm z_0)=exp\{i\beta(\pm z_0,\pm z_0)\},\\
&\beta(z,\pm z_0)=\mp\upsilon(\pm z_0)\log\big(z\pm(z_0-1)\big)+\int_{-z_0}^{z_{0}}\frac{\upsilon(s)-\chi_\pm(s)\upsilon(\pm z_{0})}{s-z}ds,
\end{align}
\begin{align}
\chi_+(z)=\left\{\begin{aligned}
&1, \quad z\in [z_0-1,z_0],\\
&0, \quad other,
\end{aligned}\right.~~
\chi_-(z)=\left\{\begin{aligned}
&1, \quad z\in [-z_0,-z_0+1],\\
&0, \quad other.
\end{aligned}\right.
\end{align}
\end{itemize}

\subsection{The first deformation of the basic RH problem}

We introduce a new matrix-valued function $M^{(1)}(z)=M(z)\delta(z)^{-\sigma_3}$ which satisfies the following matrix RH problem.

\begin{RHP}
Find a matrix value function $M^{(1)}(z)$ admiting:
\begin{itemize}
  \item Analyticity: $M^{(1)}$ is analytic on $\mathbb{C}\backslash\mathbb{R}$;
  \item Jump condition: The boundary values $M^{(1)}_{\pm}(z)$ satisfy the jump condition $M^{(1)}_{+}(z)=M^{(1)}_{-}(z)J^{(1)}(z)$ as $z\in\mathbb{R}$ where
  \begin{align}
      J^{(1)}=\left\{\begin{aligned}
&\left(
   \begin{array}{cc}
     1 & -\bar{r}\delta^2e^{-2it\theta} \\
     0 & 1 \\
   \end{array}
 \right)\left(
          \begin{array}{cc}
            1 & 0 \\
            r\delta^{-2}e^{2it\theta} & 1 \\
          \end{array}
        \right), \quad |z|>z_0,\\
&\left(
   \begin{array}{cc}
     1 & 0 \\
     \frac{r}{1-|r|^2}\delta_-^{-2}e^{2it\theta} & 1 \\
   \end{array}
 \right)\left(
          \begin{array}{cc}
            1 & -\frac{\bar{r}}{1-|r|^2}\delta_+^{2}e^{-2it\theta} \\
            0 & 1 \\
          \end{array}
        \right), \quad |z|<z_0,
\end{aligned}\right.
\end{align}
  \item asymptotic property: $M^{(1)}(z)\rightarrow \mathbb{I}$ as $z\rightarrow\infty$;
\end{itemize}
\end{RHP}

Then, we consider an analytic continuation of the jump matrix $J^{(1)}(z)$ away from the real axis. In order to achieve this objective, we first introduce the following new contours
\begin{align}
\begin{split}
\Sigma_{m}=&z_0+e^{\frac{(2m-1)}{4}\pi i}\mathbb{R}_{+},~~m=1,4,\\
\Sigma_{m}=&z_0+e^{\frac{(2m-1)}{4}\pi i}h,~~h\in(0,\frac{1}{\sqrt{2}}z_0),~~m=2,3,\\
\Sigma_{m}=&-z_0+e^{\frac{(2m-1)}{4}\pi i}h,~~h\in(0,\frac{1}{\sqrt{2}}z_0),~~m=5,8,\\
\Sigma_{m}=&-z_0+e^{\frac{(2m-1)}{4}\pi i}\mathbb{R}_{+},~~m=6,7,\\
\Sigma_{m}=&e^{\frac{(2m-1)}{4}\pi i}h,~~h\in(0,\frac{1}{\sqrt{2}}z_0),~~m=9,10,11,12,\\
\Sigma^{(2)}=&\Sigma_{1}\cup\Sigma_{2}\cup\ldots\cup\Sigma_{12},~~~~
\Sigma_{+}^{(2)}=\Sigma_{1}\cup\Sigma_{2}\cup\Sigma_{3}\cup\Sigma_{4},\\
\Sigma_{-}^{(2)}=&\Sigma_{5}\cup\Sigma_{6}\cup\Sigma_{7}\cup\Sigma_{8},~~~~
\Sigma_{0}^{(2)}=\Sigma_{9}\cup\Sigma_{10}\cup\Sigma_{11}\cup\Sigma_{12}.
\end{split}
\end{align}

The real axis $\mathbb{R}$ and $\Sigma^{(2)}$ divide the complex plane $\mathbb{C}$ into ten parts $\Omega_m$, $m=1,2,\ldots,10.$ which can be found in Fig~\ref{fig2}.

\begin{center}
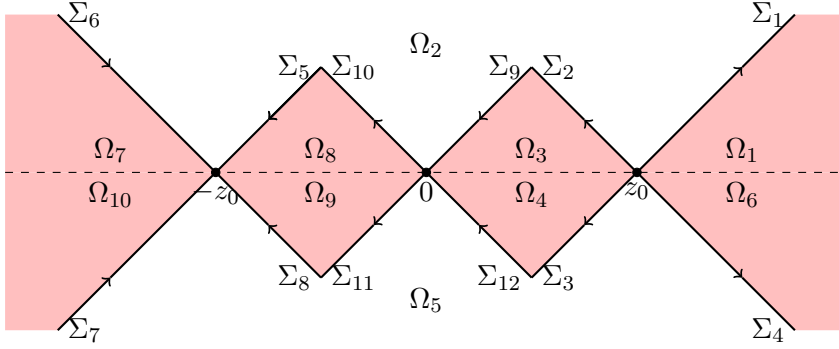
\begin{figure}
\begin{tikzpicture}[scale=0.7]
\path [fill=pink] (-4,0) -- (-2,2) to (0,0) -- (-2,-2);
\path [fill=pink] (4,0) -- (2,2) to (0,0) -- (2,-2);
\path [fill=pink] (4,0) -- (7,3) to (8,3) -- (8,0);
\path [fill=pink] (4,0) -- (7,-3) to (8,-3) -- (8,0);
\path [fill=pink] (-4,0) -- (-7,3) to (-8,3) -- (-8,0);
\path [fill=pink] (-4,0) -- (-7,-3) to (-8,-3) -- (-8,0);
\draw [dashed](-8,0)--(8,0);
\draw[->][thick](4,0)--(6,2);
\draw[-][thick](6,2)--(7,3);
\draw[->][thick](4,0)--(6,-2);
\draw[-][thick](6,-2)--(7,-3);
\draw[->][thick](4,0)--(3,1);
\draw[-][thick](3,1)--(2,2);
\draw[->][thick](4,0)--(3,-1);
\draw[-][thick](3,-1)--(2,-2);
\draw[->][thick](2,2)--(1,1);
\draw[->][thick](1,1)--(-1,-1);
\draw[->][thick](2,-2)--(1,-1);
\draw[->][thick](1,-1)--(-1,1);
\draw[-][thick](-1,1)--(-2,2);
\draw[->][thick](-2,2)--(-3,1);
\draw[-][thick](-3,1)--(-6,-2);
\draw[->][thick](-2,2)--(-3,1);
\draw[->][thick](-7,3)--(-6,2);
\draw[-][thick](-1,-1)--(-2,-2);
\draw[->][thick](-2,-2)--(-3,-1);
\draw[-][thick](-3,-1)--(-6,2);
\draw[->][thick](-7,-3)--(-6,-2);
\draw[fill] (0,0)node[below]{$0$} circle [radius=0.08];
\draw[fill] (4,0)node[below]{$z_{0}$} circle [radius=0.08];
\draw[fill] (-4,0)node[below]{$-z_{0}$} circle [radius=0.08];
\draw[fill] (2,0)node[below]{$\Omega_{4}$};
\draw[fill] (2,0)node[above]{$\Omega_{3}$};
\draw[fill] (-2,0)node[below]{$\Omega_{9}$};
\draw[fill] (-2,0)node[above]{$\Omega_{8}$};
\draw[fill] (6,0)node[below]{$\Omega_{6}$};
\draw[fill] (6,0)node[above]{$\Omega_{1}$};
\draw[fill] (-6,0)node[below]{$\Omega_{10}$};
\draw[fill] (-6,0)node[above]{$\Omega_{7}$};
\draw[fill] (0,2)node[above]{$\Omega_{2}$};
\draw[fill] (0,-2)node[below]{$\Omega_{5}$};
\draw[fill] (7,3)node[left]{$\Sigma_{1}$};
\draw[fill] (7,-3)node[left]{$\Sigma_{4}$};
\draw[fill] (-7,3)node[right]{$\Sigma_{6}$};
\draw[fill] (-7,-3)node[right]{$\Sigma_{7}$};
\draw[fill] (-2,2)node[left]{$\Sigma_{5}$};
\draw[fill] (-2,2)node[right]{$\Sigma_{10}$};
\draw[fill] (-2,-2)node[left]{$\Sigma_{8}$};
\draw[fill] (-2,-2)node[right]{$\Sigma_{11}$};
\draw[fill] (2,2)node[left]{$\Sigma_{9}$};
\draw[fill] (2,2)node[right]{$\Sigma_{2}$};
\draw[fill] (2,-2)node[left]{$\Sigma_{12}$};
\draw[fill] (2,-2)node[right]{$\Sigma_{3}$};
\end{tikzpicture}
\caption{The new contours and parts.}\label{fig2}
\end{figure}
\end{center}

\begin{prop}
There exist functions $R_{j}: {\Omega}_{m} \rightarrow \mathbb{C}, m= 1, 3, 4, 6, 7, 8, 9, 10$  such that
\begin{align*}
&R_{1}(z)=\left\{\begin{aligned}&r(z)\delta^{-2}(z), ~~~~z\in(z_{0}, +\infty),\\
&r(z_{0})\delta_{0}^{-2}(z_{0})(z-z_{0})^{-2i\upsilon(z_{0})}, ~~~~z\in\Sigma_{1},
\end{aligned}\right.\\
&R_{3}(z)=\left\{\begin{aligned}&\frac{\bar{r}(z)}{1-|r(z)|^{2}}\delta_{+}^{2}(z), ~~~~z\in(0, z_{0}),\\
&\frac{\bar{r}(z_{0})}{1-|r(z_{0})|^{2}}\delta_{0}^{2}(z_{0}) (z-z_{0})^{2i\upsilon(z_{0})}, ~~~~z\in\Sigma_{2},
\end{aligned}\right.\\
&R_{4}(z)=\left\{\begin{aligned}&\frac{r(z)}{1-|r(z)|^{2}}\delta_{-}^{-2}(z), ~~~~z\in(0, z_{0}),\\
&\frac{r(z_{0})}{1-|r(z_{0})|^{2}}\delta_{0}^{-2}(z_{0}) (z-z_{0})^{-2i\upsilon(z_{0})}, ~~~~z\in\Sigma_{3},
\end{aligned}\right.\\
&R_{6}(z)=\left\{\begin{aligned}&\bar{r}(z)\delta^{2}(z), ~~~~z\in(z_{0}, +\infty),\\
&\bar{r}(z_{0})\delta_{0}^{2}(z_{0})(z-z_{0})^{2i\upsilon(z_{0})}, ~~~~z\in\Sigma_{4},
\end{aligned}\right.\\
&R_{7}(z)=\left\{\begin{aligned}&r(z)\delta^{-2}(z), ~~~~z\in(-\infty, z_{0}),\\
&r(-z_{0})\delta_{0}^{-2}(-z_{0})(z+z_{0})^{-2i\upsilon(-z_{0})}, ~~~~z\in\Sigma_{6},
\end{aligned}\right.\\
&R_{8}(z)=\left\{\begin{aligned}&\frac{\bar{r}(z)}{1-|r(z)|^{2}}\delta_{+}^{2}(z), ~~~~z\in(-z_{0},0),\\
&\frac{\bar{r}(-z_{0})}{1-|r(-z_{0})|^{2}}\delta_{0}^{2}(-z_{0}) (z+z_{0})^{2i\upsilon(-z_{0})}, ~~~~z\in\Sigma_{5},
\end{aligned}\right.\\
&R_{9}(z)=\left\{\begin{aligned}&\frac{r(z)}{1-|r(z)|^{2}}\delta_{-}^{-2}(z), ~~~~z\in(-z_{0},0),\\
&\frac{r(-z_{0})}{1-|r(-z_{0})|^{2}}\delta_{0}^{-2}(-z_{0}) (z+z_{0})^{-2i\upsilon(-z_{0})},~~~~z\in\Sigma_{8},
\end{aligned}\right.\\
&R_{10}(z)=\left\{\begin{aligned}&\bar{r}(z)\delta^{2}(z), ~~~~z\in(-\infty, -z_{0}),\\
&\bar{r}(-z_{0})\delta_{0}^{2}(-z_{0})(z+z_{0})^{2i\upsilon(-z_{0})},~~~~z\in\Sigma_{7},
\end{aligned}\right.
\end{align*}
which satisfy the following estimate
\begin{align*}
&|R_{j}(z)|\leq c_{1}\sin^{2}(\arg (z-z_{0}))+c_{2}\left<Rez\right>^{-1/2},~~~~j=1,3,4,6,\\
&|R_{j}(z)|\leq c_{1}\sin^{2}(\arg (z+z_{0}))+c_{2}\left<Rez\right>^{-1/2},~~~~j=7,8,9,10,\\
&|\bar{\partial}R_{j}(z)|\leq c_{1}\bar{\partial}\chi_{\mathcal{Z}}(z)+c_{2}|p'_{j}(Rez)|+c_{3}|z-z_{0}|^{-1/2},~~~~j=1,3,4,6,\\
&|\bar{\partial}R_{j}(z)|\leq c_{1}\bar{\partial}\chi_{\mathcal{Z}}(z)+c_{2}|p'_{j}(Rez)|+c_{3}|z+z_{0}|^{-1/2},~~~~j=7,8,9,10,\\
&\bar{\partial}R_{j}(z)=0,z\in \Omega_{2}\cup\Omega_{5},
\end{align*}
where
\begin{align*}
&\left<Rez\right>=\sqrt{1+(Rez)^{2}},\\
&p_{1}=p_{7}=r(z),~~p_{6}=p_{10}=\bar{r}(z),\\
&p_{3}=p_{8}=\frac{r(z)}{1-|r(z)|^{2}},~~p_{4}=p_{9}=\frac{\bar{r}}{1-|r(z)|^{2}}.
\end{align*}
\end{prop}

\subsection{The construction of the mixed $\bar{\partial}$-RH problem}

We introduce a transform
\begin{align}
M^{(2)}(z)=M^{(1)}(z)R^{(2)}(z)
\end{align}
to extend the jump matrix to a new contour of oscillation term, where
\begin{align}
R^{(2)}=\left\{\begin{aligned}
&\left(
  \begin{array}{cc}
    1 & 0  \\
    -R_{1}e^{2it\theta} & 1 \\
  \end{array}
\right), ~~&z\in\Omega_{1},\\
&\left(
  \begin{array}{cc}
    1 & R_{3}e^{-2it\theta} \\
    0 & 1 \\
  \end{array}
\right), ~~&z\in\Omega_{3},\\
&\left(
  \begin{array}{cc}
    1 & 0 \\
    R_{4}e^{2it\theta} & 1 \\
  \end{array}
\right),~~ &z\in\Omega_{4},\\
&\left(
  \begin{array}{cc}
    1 & -R_{6}e^{-2it\theta} \\
    0 & 1 \\
  \end{array}
\right), ~~&z\in\Omega_{6},\\
&\left(
  \begin{array}{cc}
    1 & 0 \\
    -R_{7}e^{2it\theta} & 1 \\
  \end{array}
\right), ~~&z\in\Omega_{7},\\
&\left(
  \begin{array}{cc}
    1 & R_{8}e^{-2it\theta} \\
    0 & 1 \\
  \end{array}
\right), ~~&z\in\Omega_{8},\\
&\left(
  \begin{array}{cc}
    1 & 0 \\
    R_{9}e^{2it\theta} & 1 \\
  \end{array}
\right), ~~&z\in\Omega_{9},\\
&\left(
  \begin{array}{cc}
    1 & -R_{10}e^{-2it\theta} \\
    0 & 1 \\
  \end{array}
\right), ~~&z\in\Omega_{10},\\
&\left(
  \begin{array}{cc}
    1 & 0 \\
    0 & 1 \\
  \end{array}
\right),~~ &z\in\Omega_{2}\cup\Omega_{5}.
\end{aligned}
\right.
\end{align}

Based on the above construction, we can infer that $M^{(2)}$ have no jump on the real axis. Additionally, the influence of $\bar{\partial}$-contribution on the long-time asymptotic behaviors of the solutions can be ignored by controlling the norm of $R^{(2)}$.

Furthermore, we get the mixed $\bar{\partial}$-RH problem for $M^{(2)}(z)$ which satisfies
\begin{RHP}\label{RHP3}
Find a matrix value function $M^{(2)}$ admiting
\begin{itemize}
\item Analyticity:~$M^{(2)}(z)$ is continuous in $\mathbb{C}\setminus\Sigma^{(2)}$.
  \item Jump condition:
  \begin{align}
  M_{+}^{(2)}(x,t,z)=M_{-}^{(2)}(x,t,z)J^{(2)}(x,t,z),~~~z\in\Sigma^{(2)},
  \end{align}
  where
  \begin{align}\label{J2}
  J^{(2)}(x,t,z)=\left\{\begin{aligned}
&\left(
  \begin{array}{cc}
    1 & 0  \\
     R_{1}e^{2it\theta} & 1 \\
  \end{array}
\right), ~~&z\in\Sigma_{1},\\
&\left(
  \begin{array}{cc}
    1 &  -R_{3}e^{-2it\theta} \\
    0 & 1 \\
  \end{array}
\right), ~~&z\in\Sigma_{2}\cup\Sigma_{9},\\
&\left(
  \begin{array}{cc}
    1 & 0  \\
    -R_{4}e^{2it\theta} & 1 \\
  \end{array}
\right), ~~&z\in\Sigma_{3}\cup\Sigma_{12},\\
&\left(
  \begin{array}{cc}
    1 &  R_{6}e^{-2it\theta} \\
    0 & 1 \\
  \end{array}
\right), ~~&z\in\Sigma_{4},\\
&\left(
  \begin{array}{cc}
    1 &  -R_{8}e^{-2it\theta} \\
    0 & 1 \\
  \end{array}
\right), ~~&z\in\Sigma_{5}\cup\Sigma_{10},\\
&\left(
  \begin{array}{cc}
    1 & 0  \\
     R_{7}e^{2it\theta} & 1 \\
  \end{array}
\right), ~~&z\in\Sigma_{6},\\
&\left(
  \begin{array}{cc}
    1 &  R_{10}e^{-2it\theta} \\
    0 & 1 \\
  \end{array}
\right), ~~&z\in\Sigma_{7},\\
&\left(
  \begin{array}{cc}
    1 & 0  \\
     -R_{9}e^{2it\theta} & 1 \\
  \end{array}
\right), ~~&z\in\Sigma_{8}\cup\Sigma_{11}.\\
\end{aligned}
\right.
\end{align}
  \item Asymptotic behavior:
  \begin{align}
  M^{(2)}(x,t,z)\rightarrow\mathbb{I}~~as~~z\rightarrow\infty.
  \end{align}
  \item $\bar{\partial}M^{(2)}=M^{(2)}\bar{\partial}R^{(2)}(z),$ as $z\in\mathbb{C}\setminus\Sigma^{(2)}$ where
   \begin{align}
\bar{\partial}R^{(2)}=\left\{\begin{aligned}
&\left(
  \begin{array}{cc}
    0 & 0  \\
    -\bar{\partial}R_{1}e^{2it\theta} & 0 \\
  \end{array}
\right), ~~&z\in\Omega_{1},\\
&\left(
  \begin{array}{cc}
    0 & \bar{\partial}R_{3}e^{-2it\theta} \\
    0 & 0 \\
  \end{array}
\right), ~~&z\in\Omega_{3},\\
&\left(
  \begin{array}{cc}
    0 & 0 \\
    \bar{\partial}R_{4}e^{2it\theta} & 0 \\
  \end{array}
\right),~~ &z\in\Omega_{4},\\
&\left(
  \begin{array}{cc}
    0 & -\bar{\partial}R_{6}e^{-2it\theta} \\
    0 & 0 \\
  \end{array}
\right), ~~&z\in\Omega_{6},\\
&\left(
  \begin{array}{cc}
    0 & 0 \\
    -\bar{\partial}R_{7}e^{2it\theta} & 0 \\
  \end{array}
\right), ~~&z\in\Omega_{7},\\
&\left(
  \begin{array}{cc}
    0 & \bar{\partial}R_{8}e^{-2it\theta} \\
    0 & 0 \\
  \end{array}
\right), ~~&z\in\Omega_{8},\\
&\left(
  \begin{array}{cc}
    0 & 0 \\
    \bar{\partial}R_{9}e^{2it\theta} & 0 \\
  \end{array}
\right), ~~&z\in\Omega_{9},\\
&\left(
  \begin{array}{cc}
    0 & -\bar{\partial}R_{10}e^{-2it\theta} \\
    0 & 0 \\
  \end{array}
\right), ~~&z\in\Omega_{10},\\
&\left(
  \begin{array}{cc}
    0 & 0 \\
    0 & 0 \\
  \end{array}
\right),~~ &z\in\Omega_{2}\cup\Omega_{5}.
\end{aligned}
\right.
\end{align}
\end{itemize}
\end{RHP}

\section{The decomposition of the mixed $\bar{\partial}$-RH problem}\label{s:5}

In this section, we split the mixed $\bar{\partial}$-RH problem into  a model RH problem with $\bar{\partial}R^{(2)}=0$ and a pure $\bar{\partial}$-RH problem with $\bar{\partial}R^{(2)}\neq0$ to solve the RH problem~\ref{RHP3}.

\subsection{The model RH problem}
In the scheme of $\bar{\partial}R^{(2)}=0$, we construct a model RH problem.
\begin{RHP}\label{RHP4}
Find a matrix value function $M^{(2)}_{RHP}$ which admits
\begin{itemize}
\item Analyticity:~$M^{(2)}_{RHP}(z)$ is analytic in $\mathbb{C}\setminus\Sigma^{(2)}$.
  \item Jump condition:
  \begin{align}
  M_{RHP+}^{(2)}(x,t,z)=M_{RHP-}^{(2)}(x,t,z)J^{(2)}(x,t,z),~~~~z\in\Sigma^{(2)},
  \end{align}
  where $J^{(2)}(x,t,z)$ are shown in \eqref{J2}.
  \item Asymptotic behavior:
  \begin{align}
  M^{(2)}_{RHP}(x,t,z)\rightarrow\mathbb{I}~~as~~z\rightarrow\infty.
  \end{align}
\end{itemize}
\end{RHP}

Then, we aim to construct the solution $M^{(2)}_{RHP}$ for the RH problem~\ref{RHP4}.

\begin{prop}\label{prop4.1}
The jump matrix $J^{(2)}$ which defined in\eqref{J2} satisfies the following estimates
\begin{align}\label{J2}
||J^{(2)}-\mathbb{I}||_{L^{\infty}(\Sigma_{\pm}^{(2)})}=\left\{\begin{aligned}
&||J^{(2)}-\mathbb{I}||_{L^{\infty}(\Sigma_{\pm}^{(2)}\setminus\mathcal{U}_{\pm z_0})}
=O\left(e^{-\frac{\sqrt{2}|\alpha|t}{4}|z\mp z_0|(z_{0}^{-2}-|z|^{-2})}\right),\\
&||J^{(2)}-\mathbb{I}||_{L^{\infty}(\Sigma_{0}^{(2)})}
=O\left(e^{-\frac{|\alpha|t}{4z_{0}}}\right),
\end{aligned}\right.
\end{align}
where $\mathcal{U}_{\pm z_0}=\left\{z:|z\mp z_0|\leqslant \frac{z_{0}}{2}\right\}$.
\end{prop}

We note that $J^{(2)}\rightarrow\mathbb{I}$ provided that $z\rightarrow 0$, it is not necessary to study only the neighborhood of $z=0$. In term of the proposition~\ref{prop4.1}, it turns out that once we ignore the jump condition of $M^{(2)}_{RHP}(z)$, there is only a small exponential error with respect to $t$ outside of $\mathcal{U}_{+z_0}\cup\mathcal{U}_{-z_0}$. Thus, we only need to consider the appropriate  parabolic cylinder models at $\mathcal{U}_{\pm z_0}$ to derive the solution of $M^{(2)}_{RHP}(z)$.

For $z\rightarrow z_0$, we consider a scaling transformation
\begin{align}
\tilde{M}_{z_0}(k)=M^{(2)}_{RHP}(z),~~~~z=\sqrt{\frac{z_0}{-\alpha t}}k+z_0,
\end{align}
thus, we derive the following RH problem.
\begin{RHP}
Find a matrix value function $\tilde{M}_{z_0}$ which admits
\begin{itemize}
\item Analyticity:~$\tilde{M}_{z_0}$ is analytic in $\mathbb{C}\setminus\tilde{\Sigma}^{(2)}_+$, where $\tilde{\Sigma}^{(2)}_+=\Sigma^{(2)}_+-z_0$.
  \item Jump condition:
  \begin{align}
  \tilde{M}_{z_0+}(k)=\tilde{M}_{z_0-}(k)\tilde{J}_1^{(2)}(k),~~~~k\in\tilde{\Sigma}^{(2)}_+,
  \end{align}
  where $\tilde{J}_1^{(2)}(k)=J^{(2)}\big(\sqrt{\frac{z_0}{-\alpha t}}k+z_0\big)$ and can be expressed as
  \begin{align}
  \tilde{J}_1^{(2)}(k)=e^{\big(-\frac{1}{4}ik^2+\frac{1}{4}i\sqrt{\frac{z_0}{-\alpha t}}k^2-\frac{1}{4}i\frac{1-z_0^2}{z_0}k\big)\hat{\sigma}_3}k^{i\upsilon(z_0)\hat{\sigma}_3}\tilde{V}^{(2)}_1,
  \end{align}
  and
  \begin{align}
  \tilde{V}^{(2)}_1=\left\{\begin{aligned}
&\left(
  \begin{array}{cc}
    1 & 0  \\
    r(z_0)\delta_0^{-2}(z_0)\big(-\frac{\alpha t}{z_0}\big)^{i\upsilon(z_0)}e^{-\frac{i\alpha t}{z_0}} & 1 \\
  \end{array}
\right), ~~&z\in\tilde{\Sigma}_{1},\\
&\left(
  \begin{array}{cc}
    1 & -\frac{\bar{r}(z_0)}{1-|r(z_0)|^2}\delta_0^{2}(z_0)\big(-\frac{\alpha t}{z_0}\big)^{-i\upsilon(z_0)}e^{\frac{i\alpha t}{z_0}} \\
    0 & 1 \\
  \end{array}
\right), ~~&z\in\tilde{\Sigma}_{2},\\
&\left(
  \begin{array}{cc}
    1 & 0 \\
    -\frac{r(z_0)}{1-|r(z_0)|^2}\delta_0^{-2}(z_0)\big(-\frac{\alpha t}{z_0}\big)^{i\upsilon(z_0)}e^{-\frac{i\alpha t}{z_0}} & 1 \\
  \end{array}
\right),~~ &z\in\tilde{\Sigma}_{3},\\
&\left(
  \begin{array}{cc}
    1 & \bar{r}(z_0)\delta_0^{2}(z_0)\big(-\frac{\alpha t}{z_0}\big)^{-i\upsilon(z_0)}e^{\frac{i\alpha t}{z_0}} \\
    0 & 1 \\
  \end{array}
\right), ~~&z\in\tilde{\Sigma}_{4},
\end{aligned}
\right.
\end{align}
where $\tilde{\Sigma}_{j}=\Sigma_{j}-z_0, j=1,2,3,4.$
  \item Asymptotic behavior:
  \begin{align}
  \tilde{M}_{z_0}\rightarrow\mathbb{I}~~as~~k\rightarrow\infty.
  \end{align}
\end{itemize}
\end{RHP}

Then, we choose
\begin{align}
N_{z_0}(k)=\tilde{M}_{z_0}(k)e^{\big(-\frac{1}{4}ik^2+\frac{1}{4}i\sqrt{\frac{z_0}{-\alpha t}}k^2-\frac{1}{4}i\frac{1-z_0^2}{z_0}k\big)\sigma_3}k^{i\upsilon(z_0)\sigma_3}\tilde{V}_1,
\end{align}
where
\begin{align}
  \tilde{V}_1=\left\{\begin{aligned}
&\left(
  \begin{array}{cc}
    1 & 0  \\
    r(z_0)\delta_0^{-2}(z_0)\big(-\frac{\alpha t}{z_0}\big)^{i\upsilon(z_0)}e^{-\frac{i\alpha t}{z_0}} & 1 \\
  \end{array}
\right), ~~&z\in\tilde{\Omega}_{1},\\
&\left(
  \begin{array}{cc}
    1 & -\frac{\bar{r}(z_0)}{1-|r(z_0)|^2}\delta_0^{2}(z_0)\big(-\frac{\alpha t}{z_0}\big)^{-i\upsilon(z_0)}e^{\frac{i\alpha t}{z_0}} \\
    0 & 1 \\
  \end{array}
\right), ~~&z\in\tilde{\Omega}_{3},\\
&\left(
  \begin{array}{cc}
    1 & 0 \\
    -\frac{r(z_0)}{1-|r(z_0)|^2}\delta_0^{-2}(z_0)\big(-\frac{\alpha t}{z_0}\big)^{i\upsilon(z_0)}e^{-\frac{i\alpha t}{z_0}} & 1 \\
  \end{array}
\right),~~ &z\in\tilde{\Omega}_{4},\\
&\left(
  \begin{array}{cc}
    1 & \bar{r}(z_0)\delta_0^{2}(z_0)\big(-\frac{\alpha t}{z_0}\big)^{-i\upsilon(z_0)}e^{\frac{i\alpha t}{z_0}} \\
    0 & 1 \\
  \end{array}
\right), ~~&z\in\tilde{\Omega}_{6},\\
&\left(
  \begin{array}{cc}
    1 & 0 \\
    0 & 1 \\
  \end{array}
\right), ~~&z\in\tilde{\Omega}_{2}\cup\tilde{\Omega}_{5}.
\end{aligned}
\right.
\end{align}
The real axis $\mathbb{R}$ and $\tilde{\Sigma}^{(2)}_+$ divide the complex plane $\mathbb{C}$ into six parts $\tilde{\Omega}_m$, $m=1,2,\ldots,6.$ which can be found in Fig~\ref{fig3}.

\begin{center}
\begin{figure}
\begin{tikzpicture}[scale=0.7]
\draw [dashed](0,0)--(8,0);
\draw[->][thick](4,0)--(6,2);
\draw[-][thick](6,2)--(7,3);
\draw[->][thick](4,0)--(6,-2);
\draw[-][thick](6,-2)--(7,-3);
\draw[->][thick](4,0)--(3,1);
\draw[-][thick](3,1)--(1,3);
\draw[->][thick](4,0)--(3,-1);
\draw[-][thick](3,-1)--(1,-3);
\draw[fill] (4,0)node[below]{$0$} circle [radius=0.08];
\draw[fill] (2,0)node[below]{$\tilde{\Omega}_{4}$};
\draw[fill] (2,0)node[above]{$\tilde{\Omega}_{3}$};
\draw[fill] (6,0)node[below]{$\tilde{\Omega}_{6}$};
\draw[fill] (6,0)node[above]{$\tilde{\Omega}_{1}$};
\draw[fill] (4,2)node[above]{$\tilde{\Omega}_{2}$};
\draw[fill] (4,-2)node[below]{$\tilde{\Omega}_{5}$};
\draw[fill] (7,3)node[left]{$\tilde{\Sigma}_{1}$};
\draw[fill] (7,-3)node[left]{$\tilde{\Sigma}_{4}$};
\draw[fill] (1,3)node[right]{$\tilde{\Sigma}_{2}$};
\draw[fill] (1,-3)node[right]{$\tilde{\Sigma}_{3}$};
\end{tikzpicture}
\caption{The new contours $\tilde{\Sigma}_{j}$, $j=1, \cdots, 4$ and parts $\tilde{\Omega}_{m}$, $m=1, 2\cdots, 6$.}\label{fig3}
\end{figure}
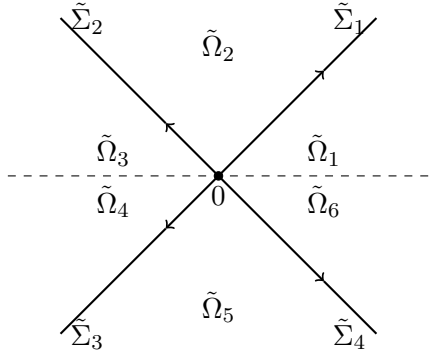
\end{center}

It is evident that $N_{z_0}(k)$ solves the following RH problem.
\begin{RHP}\label{RHP6}
Find a matrix value function $N_{z_0}(k)$ admitting
\begin{itemize}
\item Analyticity:~$N_{z_0}(k)$ is analytic in $\mathbb{C}\setminus\mathbb{R}$.
  \item Jump condition:
  \begin{align}
  N_{z_0+}(k)=N_{z_0-}(k)V_{z_0},~~~~k\in\mathbb{R},
  \end{align}
  where
  \begin{align}
  V_{z_0}=
  \left(
  \begin{array}{cc}
    1-|r(z_0)|^2 & -\bar{r}(z_0)\delta_0^2(z_0)\big(-\frac{\alpha t}{z_0}\big)^{-i\upsilon(z_0)}e^{\frac{i\alpha t}{z_0}}  \\
    {r}(z_0)\delta_0^{-2}(z_0)\big(-\frac{\alpha t}{z_0}\big)^{i\upsilon(z_0)}e^{-\frac{i\alpha t}{z_0}} & 1 \\
  \end{array}
\right), ~~&k\in\mathbb{R}.
\end{align}
  \item Asymptotic behavior:
  \begin{align}
  N_{z_0}(k)e^{-\big(-\frac{1}{4}ik^2+\frac{1}{4}i\sqrt{\frac{z_0}{-\alpha t}}k^2-\frac{1}{4}i\frac{1-z_0^2}{z_0}k\big)\sigma_3}k^{-i\upsilon(z_0)\sigma_3}
  \rightarrow\mathbb{I}~~as~~k\rightarrow\infty.
  \end{align}
\end{itemize}
\end{RHP}

For the case where $z$ is centered at $z_0$, we can obtain the solution of the RH problem~\ref{RHP6} for the $N_{z_0}(k)$ in terms of the parabolic cylinder model. We choose
\begin{align*}
N_{z_0}(k)=\left(
                \begin{array}{cc}
                  N_{z_0,11}(k) & N_{z_0,12}(k) \\
                  N_{z_0,21}(k) & N_{z_0,22}(k) \\
                \end{array}
              \right),
\end{align*}
which admits
\begin{align}
N_{z_0,11}(k)&= \left\{\begin{aligned}
e^{\frac{1}{4}\pi \upsilon(z_0)}D_{-a_1}(e^{-\frac{1}{4}\pi i}k),~~&Im~k>0,\\
e^{-\frac{3}{4}\pi \upsilon(z_0)}D_{-a_1}(e^{\frac{3}{4}\pi i}k),~~&Im~k<0,
\end{aligned}
\right.\\
N_{z_0,12}(k)&=\left\{\begin{aligned}
\frac{1}{\beta_{21}^{z_0}}e^{-\frac{3}{4}\pi \upsilon(z_0)}\bigg[\frac{d}{dk}D_{a_1}(e^{-\frac{3}{4}\pi i}k)+\frac{1}{2}ikD_{a_1}(e^{-\frac{3}{4}\pi i}k)\bigg],~~&Im~k>0,\\
\frac{1}{\beta_{21}^{z_0}}e^{\frac{1}{4}\pi \upsilon(z_0)}\bigg[\frac{d}{dk}D_{a_1}(e^{\frac{1}{4}\pi i}k)+\frac{1}{2}ikD_{a_1}(e^{\frac{1}{4}\pi i}k)\bigg],~~&Im~k>0,
\end{aligned}
\right.\\
N_{z_0,21}(k)&=\left\{\begin{aligned}
\frac{1}{\beta_{12}^{z_0}}e^{\frac{1}{4}\pi \upsilon(z_0)}\bigg[\frac{d}{dk}D_{-a_1}(e^{-\frac{1}{4}\pi i}k)-\frac{1}{2}ikD_{-a_1}(e^{-\frac{1}{4}\pi i}k)\bigg],~~&Im~k>0,\\
\frac{1}{\beta_{12}^{z_0}}e^{-\frac{3}{4}\pi \upsilon(z_0)}\bigg[\frac{d}{dk}D_{-a_1}(e^{\frac{3}{4}\pi i}k)-\frac{1}{2}ikD_{-a_1}(e^{\frac{3}{4}\pi i}k)\bigg],~~&Im~k>0,
\end{aligned}
\right.\\
N_{z_0,22}(k)&= \left\{\begin{aligned}
e^{-\frac{3}{4}\pi \upsilon(z_0)}D_{a_1}(e^{-\frac{3}{4}\pi i}k),~~&Im~k>0,\\
e^{\frac{1}{4}\pi \upsilon(z_0)}D_{a_1}(e^{\frac{1}{4}\pi i}k),~~&Im~k<0,
\end{aligned}
\right.
\end{align}
where
\begin{align}\label{beta1}
\beta_{12}^{z_0}=-\frac{\sqrt{2\pi}e^{-\frac{1}{2}\pi \upsilon(z_0)}e^{-\frac{\pi i}{4}}}{\bar{r}(z_0)\Gamma(a_1)},~~\beta_{21}^{z_0}=\frac{\upsilon(z_0)}{\beta_{12}^{z_0}},~~a_1=i\upsilon(z_0).
\end{align}

Similarly, for $z\rightarrow-z_0$, we consider a scaling transformation
\begin{align}
\tilde{M}_{-z_0}(k)=M^{(2)}_{RHP}(z),~~~~z=\sqrt{\frac{z_0}{-\alpha t}}k-z_0,
\end{align}
thus, we derive the following RH problem.
\begin{RHP}
Find a matrix value function $\tilde{M}_{-z_0}$ which admits
\begin{itemize}
\item Analyticity:~$\tilde{M}_{-z_0}$ is analytic in $\mathbb{C}\setminus\tilde{\Sigma}^{(2)}_-$, where $\tilde{\Sigma}^{(2)}_-=\Sigma^{(2)}_-+z_0$.
  \item Jump condition:
  \begin{align}
  \tilde{M}_{-z_0+}(k)=\tilde{M}_{-z_0-}(k)\tilde{J}_2^{(2)}(k),~~~~k\in\tilde{\Sigma}^{(2)}_-,
  \end{align}
  where $\tilde{J}_2^{(2)}(k)=J^{(2)}\big(\sqrt{\frac{z_0}{-\alpha t}}k-z_0\big)$ and can be expressed as
  \begin{align}
  \tilde{J}_2^{(2)}(k)=e^{\big(\frac{1}{4}ik^2-\frac{1}{4}i\sqrt{\frac{z_0}{-\alpha t}}k^2-\frac{1}{4}i\frac{1-z_0^2}{z_0}k\big)\hat{\sigma}_3}k^{i\upsilon(-z_0)\hat{\sigma}_3}\tilde{V}^{(2)}_2,
  \end{align}
  and
  \begin{align}
  \tilde{V}^{(2)}_2=\left\{\begin{aligned}
&\left(
  \begin{array}{cc}
    1 & -\frac{\bar{r}(-z_0)}{1-|r(-z_0)|^2}\delta_0^{2}(-z_0)\big(-\frac{\alpha t}{z_0}\big)^{-i\upsilon(-z_0)}e^{\frac{i\alpha t}{z_0}}  \\
    0 & 1 \\
  \end{array}
\right), ~~&z\in\tilde{\Sigma}_{5},\\
&\left(
  \begin{array}{cc}
    1 & 0 \\
    r(-z_0)\delta_0^{-2}(-z_0)\big(-\frac{\alpha t}{z_0}\big)^{i\upsilon(-z_0)}e^{-\frac{i\alpha t}{z_0}} & 1 \\
  \end{array}
\right), ~~&z\in\tilde{\Sigma}_{6},\\
&\left(
  \begin{array}{cc}
    1 & \bar{r}(-z_0)\delta_0^{2}(-z_0)\big(-\frac{\alpha t}{z_0}\big)^{-i\upsilon(-z_0)}e^{\frac{i\alpha t}{z_0}} \\
   0 & 1 \\
  \end{array}
\right),~~ &z\in\tilde{\Sigma}_{7},\\
&\left(
  \begin{array}{cc}
    1 & 0 \\
     -\frac{r(-z_0)}{1-|r(-z_0)|^2}\delta_0^{-2}(-z_0)\big(-\frac{\alpha t}{z_0}\big)^{i\upsilon(-z_0)}e^{-\frac{i\alpha t}{z_0}} & 1 \\
  \end{array}
\right), ~~&z\in\tilde{\Sigma}_{8},
\end{aligned}
\right.
\end{align}
where $\tilde{\Sigma}_{j}=\Sigma_{j}+z_0, j=5,6,7,8.$
  \item Asymptotic behavior:
  \begin{align}
  \tilde{M}_{-z_0}\rightarrow\mathbb{I}~~as~~k\rightarrow\infty.
  \end{align}
\end{itemize}
\end{RHP}

Then, we choose
\begin{align}
N_{-z_0}(k)=\tilde{M}_{-z_0}(k)e^{\big(\frac{1}{4}ik^2-\frac{1}{4}i\sqrt{\frac{z_0}{-\alpha t}}k^2-\frac{1}{4}i\frac{1-z_0^2}{z_0}k\big)\sigma_3}k^{i\upsilon(-z_0)\sigma_3}\tilde{V}_2,
\end{align}
where
\begin{align}
  \tilde{V}_2=\left\{\begin{aligned}
&\left(
  \begin{array}{cc}
    1 & -\frac{\bar{r}(-z_0)}{1-|r(-z_0)|^2}\delta_0^{2}(-z_0)\big(-\frac{\alpha t}{z_0}\big)^{-i\upsilon(-z_0)}e^{\frac{i\alpha t}{z_0}}  \\
    0 & 1 \\
  \end{array}
\right), ~~&z\in\tilde{\Omega}_{8},\\
&\left(
  \begin{array}{cc}
    1 & 0 \\
    r(-z_0)\delta_0^{-2}(-z_0)\big(-\frac{\alpha t}{z_0}\big)^{i\upsilon(-z_0)}e^{-\frac{i\alpha t}{z_0}} & 1 \\
  \end{array}
\right), ~~&z\in\tilde{\Omega}_{7},\\
&\left(
  \begin{array}{cc}
    1 & \bar{r}(-z_0)\delta_0^{2}(-z_0)\big(-\frac{\alpha t}{z_0}\big)^{-i\upsilon(-z_0)}e^{\frac{i\alpha t}{z_0}} \\
   0 & 1 \\
  \end{array}
\right),~~ &z\in\tilde{\Omega}_{10},\\
&\left(
  \begin{array}{cc}
    1 & 0 \\
     -\frac{r(-z_0)}{1-|r(-z_0)|^2}\delta_0^{-2}(-z_0)\big(-\frac{\alpha t}{z_0}\big)^{i\upsilon(-z_0)}e^{-\frac{i\alpha t}{z_0}} & 1 \\
  \end{array}
\right), ~~&z\in\tilde{\Omega}_{9},\\
&\left(
  \begin{array}{cc}
    1 & 0 \\
    0 & 1 \\
  \end{array}
\right), ~~&z\in\tilde{\Omega}_{2}\cup\tilde{\Omega}_{5}.
\end{aligned}
\right.
\end{align}
The real axis $\mathbb{R}$ and $\tilde{\Sigma}^{(2)}_-$ divide the complex plane $\mathbb{C}$ into six parts $\tilde{\Omega}_m$, $m=2,5,7,8,9,10.$ which can be found in Fig~\ref{fig4}.

\begin{center}
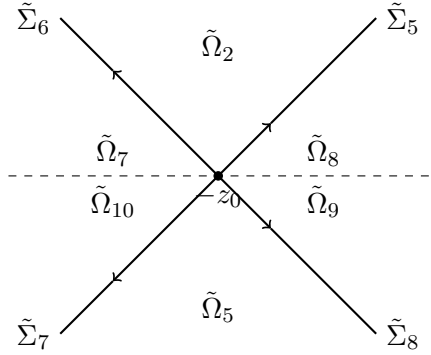
\begin{figure}
\begin{tikzpicture}[scale=0.7]
\draw [dashed](0,0)--(-8,0);
\draw[->][thick](-4,0)--(-6,2);
\draw[-][thick](-6,2)--(-7,3);
\draw[->][thick](-4,0)--(-6,-2);
\draw[-][thick](-6,-2)--(-7,-3);
\draw[->][thick](-4,0)--(-3,1);
\draw[-][thick](-3,1)--(-1,3);
\draw[->][thick](-4,0)--(-3,-1);
\draw[-][thick](-3,-1)--(-1,-3);
\draw[fill] (-4,0)node[below]{$-z_{0}$} circle [radius=0.08];
\draw[fill] (-2,0)node[below]{$\tilde{\Omega}_{9}$};
\draw[fill] (-2,0)node[above]{$\tilde{\Omega}_{8}$};
\draw[fill] (-6,0)node[below]{$\tilde{\Omega}_{10}$};
\draw[fill] (-6,0)node[above]{$\tilde{\Omega}_{7}$};
\draw[fill] (-4,2)node[above]{$\tilde{\Omega}_{2}$};
\draw[fill] (-4,-2)node[below]{$\tilde{\Omega}_{5}$};
\draw[fill] (-7,3)node[left]{$\tilde{\Sigma}_{6}$};
\draw[fill] (-7,-3)node[left]{$\tilde{\Sigma}_{7}$};
\draw[fill] (-1,3)node[right]{$\tilde{\Sigma}_{5}$};
\draw[fill] (-1,-3)node[right]{$\tilde{\Sigma}_{8}$};
\end{tikzpicture}
\caption{The new contours $\tilde{\Sigma}_{j}$, $j=5, \cdots, 8$ and parts $\tilde{\Omega}_{m}$, $m=2, 5, 7\cdots, 10$.}\label{fig4}
\end{figure}
\end{center}

It is evident that $N_{-z_0}(k)$ solves the following RH problem.
\begin{RHP}\label{RHP7}
Find a matrix value function $N_{-z_0}(k)$ which admits
\begin{itemize}
\item Analyticity:~$N_{-z_0}(k)$ is analytic in $\mathbb{C}\setminus\mathbb{R}$.
  \item Jump condition:
  \begin{align}
  N_{-z_0+}(k)=N_{-z_0-}(k)V_{-z_0},~~~~k\in\mathbb{R},
  \end{align}
  where
  \begin{align}
  V_{-z_0}=
  \left(
  \begin{array}{cc}
    1-|r(-z_0)|^2 & -\bar{r}(-z_0)\delta_0^2(-z_0)\big(-\frac{\alpha t}{z_0}\big)^{-i\upsilon(-z_0)}e^{\frac{i\alpha t}{z_0}}  \\
    r(-z_0)\delta_0^{-2}(-z_0)\big(-\frac{\alpha t}{z_0}\big)^{i\upsilon(-z_0)}e^{-\frac{i\alpha t}{z_0}} & 1 \\
  \end{array}
\right), ~~&k\in\mathbb{R}.
\end{align}
  \item Asymptotic behavior:
  \begin{align}
  N_{-z_0}(k)e^{-\big(\frac{1}{4}ik^2-\frac{1}{4}i\sqrt{\frac{z_0}{-\alpha t}}k^2-\frac{1}{4}i\frac{1-z_0^2}{z_0}k\big)\sigma_3}k^{-i\upsilon(-z_0)\sigma_3}\rightarrow\mathbb{I}~~as~~k\rightarrow\infty.
  \end{align}
\end{itemize}
\end{RHP}

For the case where $z$ is centered at $-z_0$, we can obtain the solution of the RH problem~\ref{RHP7} for the $N_{-z_0}(k)$ in terms of the parabolic cylinder model. We choose
\begin{align*}
N_{-z_0}(k)=\left(
                \begin{array}{cc}
                  N_{-z_0,11}(k) & N_{-z_0,12}(k) \\
                  N_{-z_0,21}(k) & N_{-z_0,22}(k) \\
                \end{array}
              \right),
\end{align*}
which admits
\begin{align}
N_{-z_0,11}(k)&= \left\{\begin{aligned}
e^{-\frac{3}{4}\pi \upsilon(-z_0)}D_{a_2}(e^{-\frac{3}{4}\pi i}k),~~&Im~k>0,\\
e^{\frac{1}{4}\pi \upsilon(-z_0)}D_{a_2}(e^{\frac{1}{4}\pi i}k),~~&Im~k<0,
\end{aligned}
\right.\\
N_{-z_0,12}(k)&=\left\{\begin{aligned}
\frac{1}{\beta_{21}^{-z_0}}e^{\frac{1}{4}\pi \upsilon(-z_0)}\bigg[\frac{d}{dk}D_{-a_2}(e^{-\frac{1}{4}\pi i}k)-\frac{1}{2}ikD_{-a_2}(e^{-\frac{1}{4}\pi i}k)\bigg],~~&Im~k>0,\\
\frac{1}{\beta_{21}^{-z_0}}e^{\frac{3}{4}\pi \upsilon(-z_0)}\bigg[\frac{d}{dk}D_{-a_2}(e^{\frac{3}{4}\pi i}k)-\frac{1}{2}ikD_{-a_2}(e^{\frac{3}{4}\pi i}k)\bigg],~~&Im~k>0,
\end{aligned}
\right.\\
N_{-z_0,21}(k)&=\left\{\begin{aligned}
\frac{1}{\beta_{12}^{-z_0}}e^{-\frac{3}{4}\pi \upsilon(-z_0)}\bigg[\frac{d}{dk}D_{a_2}(e^{-\frac{3}{4}\pi i}k)
+\frac{1}{2}ikD_{a_2}(e^{-\frac{3}{4}\pi i}k)\bigg],~~&Im~k>0,\\
\frac{1}{\beta_{12}^{-z_0}}e^{\frac{1}{4}\pi \upsilon(-z_0)}\bigg[\frac{d}{dk}D_{a_2}(e^{\frac{1}{4}\pi i}k)+\frac{1}{2}ikD_{a_2}(e^{\frac{1}{4}\pi i}k)\bigg],~~&Im~k>0,
\end{aligned}
\right.\\
N_{-z_0,22}(k)&= \left\{\begin{aligned}
e^{\frac{1}{4}\pi \upsilon(-z_0)}D_{-a_2}(e^{-\frac{1}{4}\pi i}k),~~&Im~k>0,\\
e^{-\frac{3}{4}\pi \upsilon(-z_0)}D_{-a_2}(e^{\frac{3}{4}\pi i}k),~~&Im~k<0,
\end{aligned}
\right.
\end{align}
where
\begin{align}\label{beta2}
\beta_{12}^{-z_0}=\frac{\sqrt{2\pi}e^{-\frac{1}{2}\pi \upsilon(-z_0)}e^{\frac{\pi i}{4}}}{\bar{r}(-z_0)\Gamma(-a_2)},~~\beta_{21}^{-z_0}=\frac{\upsilon(-z_0)}{\beta_{12}^{-z_0}},~~
a_2=i\upsilon(-z_0).
\end{align}
We define $\tilde{M}_{\pm z_0}$ as
\begin{align}
\tilde{M}_{\pm z_0}=\mathbb{I}+\frac{\tilde{M}_{\pm z_0}^{1}}{k}+O(\frac{1}{k^2}),~~~~k\rightarrow\infty.
\end{align}

Then, we have $\beta_{12}^{z_0}=-i\left(\tilde{M}_{z_0}^1\right)_{12}$, $\beta_{12}^{-z_0}=i\left(\tilde{M}_{-z_0}^1\right)_{12}$, $\left(\tilde{M}_{\pm z_0}^{1}\right)_{11}=0$ where $\left(\tilde{M}_{\pm z_0}^1\right)_{ij}$ represent the $(i,j)$-th elements of the matrices $\tilde{M}_{\pm z_0}^1$.

Hence, $M^{(2)}_{RHP}(z)$ can be written as $M^{(2)}_{RHP}(z)=\tilde{M}_{z_0}(k)+\tilde{M}_{-z_0}(k)$.

\subsection{Pure $\bar{\partial}$-RH problem}
In this section, we study a pure $\bar{\partial}$-RH problem and show the asymptotic of its solution.

Firstly, we construct a transform
\begin{align}
M^{(3)}(z)=M^{(2)}(z)M^{(2)}_{RHP}(z)^{-1}
\end{align}
and introduce the RH problem of $M^{(3)}(z)$.

\begin{RHP}\label{RHP10}
Find a matrix value function $M^{(3)}(z)$ which admits
\begin{itemize}
\item Analyticity:~$M^{(3)}$ is continuous with sectionally continuous first partial derivatives in $\mathbb{C}\backslash\Sigma^{(2)}$;
  \item Asymptotic behavior:
  \begin{align}
  M^{(3)}(x,t,z)\rightarrow\mathbb{I}~~as~~z\rightarrow\infty.
  \end{align}
  \item $\bar{\partial}M^{(3)}(z)=M^{(3)}(z)R^{(3)}(z)$ as $z\in \mathbb{C}$, where
  \begin{align}
       R^{(3)}=M_{RHP}^{(2)}(z)\bar{\partial}R^{(2)}M_{RHP}^{(2)}(z)^{-1}.
  \end{align}
\end{itemize}
\end{RHP}

The solution $M^{(3)}$ of RH problem~\ref{RHP10} can be written as
\begin{align}
M^{(3)}(z)=\mathbb{I}+\frac{1}{2\pi i}\int\int\frac{M^{(3)}W^{(3)}}{s-z}\mathrm{d}m(s),
\end{align}
where $m(s)$ is Lebesgue measure. We rewrite the above function as
\begin{align}
(\mathbb{I}-\mathrm{S})\left[M^{(3)}(z)\right]=\mathbb{I},
\end{align}
where $\mathrm{S}[f](z)=-\frac{1}{\pi}\int\int_{\mathbb{C}}\frac{f(s)W^{(3)}(s)}{s-z}\mathrm{d}m(s)$.
\begin{prop}\label{prop4.7}
For $t\rightarrow+\infty$, the operator $\mathrm{S}$ admits that
\begin{align}
||\mathrm{S}||_{L^{\infty}\rightarrow L^{\infty}}\leq ct^{-1/6},
\end{align}
where $c$ is a constant. Therefore, $(\mathbb{I}-\mathrm{S})^{-1}$ exist.
\end{prop}
\begin{proof}
We only consider the case that in the region $\Omega_1$. We choose $s=u+iv$ and derive
\begin{align}
|S[f](z)|&\leq\frac{1}{\pi}
\iint_{\Omega_{1}}\frac{|fM^{(2)}_{RHP}(s)\bar{\partial}R^{(2)}(s)M^{(2)}_{RHP}(s)^{-1}|}
{|s-z|}dm(s)\notag\\
& \leq \frac{1}{\pi}\big|\big|f\ \big|\big|_{L^{\infty}(\Omega_{1})}\big|\big|M_{RHP}^{(2)}\ \big|\big|_{L^{\infty}(\Omega_{1})}\big|\big|M_{RHP}^{(2)-}\ \big|\big|_{L^{\infty}(\Omega_{1})}
\iint_{\Omega_{1}}\frac{|\bar{\partial}R_{1}(s)|e^{\alpha tv\frac{u^2+v^2-z_0^2}{2(u^2+v^2)z_0^2}}}
{|s-z|}dudv\notag\\
&\leq c\iint_{\Omega_{1}}
\frac{|\bar{\partial}R_{1}(s)||e^{\alpha tv\frac{u^{2}+v^{2}-z_{0}^{2}}{2(u^{2}+v^{2})z_{0}^{2}}}|}{|s-z|}dudv,
\end{align}
where $c$ is a constant.
Further we can obtain that
\begin{align}
||\mathrm{S}||_{L^{\infty}\rightarrow L^{\infty}}\leq c(I_{1}+I_{2}+I_{3})\leq c|t|^{-1/6},
\end{align}
where
\begin{align}
I_{1}=\iint_{\Omega_{1}}
\frac{|\bar{\partial}\chi_{\mathcal{Z}}(s)|
e^{\alpha tv\frac{u^{2}+v^{2}-z_{0}^{2}}{2(u^{2}+v^{2})z_{0}^{2}}}}{|s-z|}df(s), \\
I_{2}=\iint_{\Omega_{1}}
\frac{|r'(Rez)|e^{\alpha tv\frac{u^{2}+v^{2}-z_{0}^{2}}{2(u^{2}+v^{2})z_{0}^{2}}}}{|s-z|}df(s),\\
I_{3}=\iint_{\Omega_{1}}
\frac{|s-z_{0}|^{-\frac{1}{2}}e^{\alpha tv\frac{u^{2}+v^{2}-z_{0}^{2}}{2(u^{2}+v^{2})z_{0}^{2}}}
}{|s-z|}df(s).
\end{align}

Firstly, we consider the estimate of $I_1$.
We set $z=\mu+i\eta$ and derive
\begin{align}
||\frac{1}{s-z}||^2_{L^2(v+z_0,+\infty)}=\int_{v+z_0}^{+\infty}\frac{1}{(u-\mu)^2+(v-\eta)^2}du=\frac{1   }{|v-\eta|}\int_{v+z_0}^{+\infty}\frac{1}{\big(\frac{u-\mu}{v-\eta}\big)^2+1}d\frac{u-\mu}{v-\eta}
\leq\frac{\pi}{|v-\eta|}.
\end{align}
According to the truth that $\frac{u^{2}+v^{2}-z_{0}^{2}}{(u^{2}+v^{2})z_{0}^{2}}>\frac{v^{2}}{(u^{2}+v^{2})z_{0}^{2}}>0$, we assume that there exist a constant $\kappa$ such that $\frac{u^{2}+v^{2}-z_{0}^{2}}{(u^{2}+v^{2})z_{0}^{2}}>\kappa$.
Further we derive
\begin{align}
\begin{split}
|I_{1}|&\leq\int_{0}^{+\infty}\int_{v+z_{0}}^{+\infty}
\frac{|\bar{\partial}\chi_{\mathcal{Z}}(s)|e^{-|\alpha|tv\frac{u^{2}+v^{2}-z_{0}^{2}}{2(u^{2}+v^{2})z_{0}^{2}}}}
{|s-z|}dudv\\
&\leq\int_{0}^{+\infty}e^{-|\alpha|tv\frac{\kappa}{2}}\big|\big|\bar{\partial}\chi_{\mathcal{Z}}(s)
\big|\big|_{L^{2}(v+z_{0})}\Big|\Big|\frac{1}{s-z}\Big|\Big|_{L^{2}(v+z_{0})}dv \\
&\leq C\bigg(\int_{0}^{\eta}e^{-|\alpha|tv\frac{\kappa}{2}}\frac{1}{\sqrt{\eta-v}}dv
+\int_{\eta}^{+\infty}e^{-|\alpha|tv\frac{\kappa}{2}}\frac{1}{\sqrt{v-\eta}}dv\bigg).
\end{split}
\end{align}
On the one hand, we employ the inequality $e^{-z}\leq z^{-1/6}$ and deduce that
\begin{align*}
\int_{0}^{\eta}e^{-|\alpha|tv\frac{\kappa}{2}}\frac{1}{\sqrt{\eta-v}}dv\lesssim t^{-\frac{1}{6}}.
\end{align*}
For the other hand, we choose $\omega=v-\eta$ and derive
\begin{align*}
\int_{\eta}^{+\infty}e^{-|\alpha|tv\frac{\kappa}{2}}\frac{1}{\sqrt{v-\eta}}dv\leq
\int_{0}^{+\infty}e^{-|\alpha|t\omega\frac{\kappa}{2}}\frac{1}{\sqrt{\omega}}d\omega\lesssim t^{-\frac{1}{2}}.
\end{align*}
Thus, we obtain $I_{1}\lesssim t^{-\frac{1}{6}}$. Similarly, we can obtain that $I_{1}\lesssim |t|^{-\frac{1}{6}}$.
Finally, we consider the estimate of $I_3$.

We consider the following norms with $p>2$ and $\frac{1}{p}+\frac{1}{q}=1$
\begin{align}
\bigg|\bigg|\frac{1}{\sqrt{|s-z_{0}|}}\bigg|\bigg|_{L^{p}}
\leq \left(\int_{v+z_{0}}^{+\infty}
\frac{1}{|u-z_{0}+iv|^{\frac{p}{2}}}du\right)^{\frac{1}{p}}
\leq cv^{\frac{1}{p}-\frac{1}{2}},
\end{align}
and
\begin{align}
\bigg|\bigg|\frac{1}{|s-z|}\bigg|\bigg|_{L^{q}}\leq c|v-\eta|^{\frac{1}{q}-1}.
\end{align}
Therefore, in combination with the above results we can infer that
\begin{align}
\begin{split}
|I_{3}|&\leq\int_{0}^{+\infty}\int_{v}^{+\infty}
\frac{e^{-|\alpha|tv\frac{\kappa}{2}}}{|s-z_{0}|^{-\frac{1}{2}}|s-z|}dudv\\
&\leq\int_{0}^{+\infty}e^{-|\alpha|tv\frac{\kappa}{2}}\bigg|\bigg|\frac{1}{\sqrt{|s-z_{0}|}}\bigg|\bigg|_{L^{p}}
\bigg|\bigg|\frac{1}{|s-z|}\bigg|\bigg|_{L^{q}}dv\\ &\leq\int_{0}^{\eta}e^{-|\alpha|tv\frac{\kappa}{2}}v^{\frac{1}{p}-\frac{1}{2}}
(\eta-v)^{\frac{1}{q}-1}dv+\int_{\eta}^{+\infty}
e^{-|\alpha|tv\frac{\kappa}{2}}v^{\frac{1}{p}-\frac{1}{2}}(v-\eta)^{\frac{1}{q}-1}dv.
\end{split}
\end{align}
Taking a similar computational idea to $I_1$, we can obtain $I_{3}\lesssim t^{-\frac{1}{2}}$.

Thus, the proposition can be proved.
\end{proof}

We define
\begin{align}
M^{(3)}=\mathbb{I}+\frac{M_1^{(3)}}{z}+O(z^{-2}),
\end{align}
and give the following proposition.
\begin{prop}
For $t\rightarrow+\infty$, $M^{(3)}_{1}(x,t)$ admits the inequality
$\big|M^{(3)}_{1}(x,t)\big|\lesssim t^{-1}$.
\end{prop}
\begin{proof}
Prove this proposition in a similar way to proposition~\ref{prop4.7}.
\end{proof}

\section{The long-time behaviors of the AB system}\label{s:6}
Based on the above discussion, we can infer that
\begin{align*}
M(z)=M^{(3)}(z)\left(\tilde{M}_{z_0}\left(\sqrt{-\frac{\alpha t}{z_0}}(z-z_0)\right)
+\tilde{M}_{-z_0}\left(\sqrt{-\frac{\alpha t}{z_0}}(z+z_0)\right)\right)R^{(2)^{-1}}(z)\delta^{\sigma_{3}}(z).
\end{align*}

In terms of $R^{(2)}(z)=I$ as $z\in\Omega_{2}$ or $z\in\Omega_{5}$, we derive the asymptotic behavior of $M(z)$
\begin{align*}
M\rightarrow I+\bigg(M^{(3)}_{1}+\sqrt{\frac{z_0}{-\alpha t}}\tilde{M}_{z_0}^{1}+\sqrt{\frac{z_0}{-\alpha t}}\tilde{M}_{-z_0}^{1}-
\big(i\int_{-z_0}^{z_{0}}\upsilon(s)ds\big)^{\sigma_3}\bigg)z^{-1}+O(z^{-2}).
\end{align*}

Thus, we deduce the following results.
\begin{thm}\label{thm1}
Let $A(x,t)$ and $B(x,t)$ be the solutions of the system\eqref{AB} whose initial values satisfy the condition \eqref{IVC}, the solutions admits the following asymptotic behaviors as $t\rightarrow+\infty$
\begin{align}
A(x,t)&=4\left[\left(M_{1}^{(3)}\right)_{12}+\left(\sqrt{\frac{z_0}{-\alpha t}}\tilde{M}_{z_0}^{1}\right)_{12}
+\left(\sqrt{\frac{z_0}{-\alpha t}}\tilde{M}_{-z_0}^{1}\right)_{12}\right]\\\nonumber
&=4\frac{\sqrt{z_0}}{\sqrt{-\alpha t}}\left(i\beta_{12}^{z_0}-i\beta_{12}^{-z_0}\right)+O(t^{-1}),\\
B(x,t)&=-\frac{4i}{\beta}\left[\left(M_{1}^{(3)}\right)_{11}+\left(\sqrt{\frac{z_0}{-\alpha t}}\tilde{M}_{z_0}^{1}\right)_{11}
+\left(\sqrt{\frac{z_0}{-\alpha t}}\tilde{M}_{-z_0}^{1}\right)_{11}-
\big(i\int_{-z_0}^{z_{0}}\upsilon(s)ds\big)\right]_t\\\nonumber
&=O(t^{-1}).
\end{align}
where $z_0=\sqrt{-\frac{\alpha t}{4x}}$, $\beta_{12}^{z_0}$ and $\beta_{12}^{-z_0}$ are defined in \eqref{beta1} and \eqref{beta2}, respectively.
\end{thm}

\noindent{\bf Acknowledgments} This work was supported by the National Natural Science Foundation of China under Grant No. 11975306, the Natural Science Foundation of Jiangsu Province under Grant No. BK20181351, the Six Talent Peaks Project in Jiangsu Province under Grant No. JY-059, and the Fundamental Research Fund for the Central Universities under the Grant Nos. 2019ZDPY07 and 2019QNA35.


\bibliographystyle{plain}

\end{document}